\documentclass[a4paper]{article}

\usepackage[utf8]{inputenc} 
\usepackage[T1]{fontenc} 
\usepackage{amsmath,amsthm,latexsym} 
\usepackage[fixamsmath]{mathtools}
\mathtoolsset{showmanualtags,mathic,centercolon} 
\usepackage{amssymb,amsfonts} 
\usepackage{placeins} 
\usepackage{enumitem}
\usepackage[svgnames]{xcolor} 
\usepackage{algpseudocode} 
\usepackage{algorithm} 
\usepackage{bm} 
\usepackage{mathdots}  
\usepackage{url}
\usepackage{graphicx}

\usepackage[margin=1.25in]{geometry} 
\usepackage[labelfont=sl,textfont=sl]{subcaption} 
\usepackage[font={small,sl},labelsep=period]{caption} 

\numberwithin{equation}{section} 
\theoremstyle{plain}
\newtheorem{theorem}{Theorem}[section] 
\newtheorem{lemma}[theorem]{Lemma} 
\newtheorem{corollary}[theorem]{Corollary}

\theoremstyle{definition}
\newtheorem{definition}[theorem]{Definition}

\newtheorem{assumption}[theorem]{Assumption}

\theoremstyle{remark}
\newtheorem{remark}[theorem]{Remark}

\newcommand{\secref}[1]{Section~\ref{#1}}
\newcommand{\defref}[1]{Definition~\ref{#1}}
\newcommand{\thmref}[1]{Theorem~\ref{#1}}
\newcommand{\lemref}[1]{Lemma~\ref{#1}}

\renewcommand{\algref}[1]{Algorithm~\ref{#1}}
\newcommand{\assref}[1]{Assumption~\ref{#1}}

\newcommand{\remref}[1]{Remark~\ref{#1}}

\newcommand{\R}{\mathbb{R}} 

\newcommand{\bigO}{\mathcal{O}} 
\DeclareMathOperator*{\argmin}{arg\,min} 
\DeclareMathOperator{\proj}{proj} 

\newcommand{\grad}{\nabla} 
\newcommand{\defeq}{:=} 

\newcommand{\bmat}[1]{\begin{bmatrix}#1\end{bmatrix}} 

\algrenewcommand\algorithmicrequire{\textbf{Input:}}
\algrenewcommand\algorithmicensure{\textbf{Output:}}

\renewcommand{\b}[1]{\bm{#1}} 

\newcommand{\bd}{\b{d}}

\newcommand{\bg}{\b{g}}

\newcommand{\br}{\b{r}}
\newcommand{\bs}{\b{s}}

\newcommand{\bu}{\b{u}}
\newcommand{\bv}{\b{v}}

\newcommand{\bx}{\b{x}}
\newcommand{\by}{\b{y}}

\newcommand{\bee}{\b{e}}  

\renewcommand{\t}[1]{\widetilde{#1}} 

\newcommand{\kappaef}{\kappa_{\textnormal{ef}}}
\newcommand{\kappaeg}{\kappa_{\textnormal{eg}}}
\newcommand{\gammainc}{\gamma_{\textnormal{inc}}}
\newcommand{\gammadec}{\gamma_{\textnormal{dec}}}
\newcommand{\flow}{f_{\textnormal{low}}}

\newcommand{\balpha}{\b{\alpha}}
\newcommand{\blambda}{\b{\lambda}}
\newcommand{\bphi}{\b{\phi}}
\newcommand{\bell}{\b{\ell}}

\begin{document}
\title{Model Construction for Convex-Constrained Derivative-Free Optimization}
\author{
Lindon Roberts\thanks{School of Mathematics and Statistics, University of Sydney, Camperdown NSW 2006, Australia (\texttt{lindon.roberts@sydney.edu.au}). 
This work was supported by the Australian Research Council Discovery Early Career Award DE240100006.}}

\date{\today}
\maketitle

\begin{abstract}
We develop a new approximation theory for linear and quadratic interpolation models, suitable for use in convex-constrained derivative-free optimization (DFO).
Most existing model-based DFO methods for constrained problems assume the ability to construct sufficiently accurate approximations via interpolation, but the standard notions of accuracy (designed for unconstrained problems) may not be achievable by only sampling feasible points, and so may not give practical algorithms.
This work extends the theory of convex-constrained linear interpolation developed in [Hough \& Roberts, \textit{SIAM J.~Optim}, 32:4 (2022), pp.~2552--2579] to the case of linear regression models and underdetermined quadratic interpolation models.
\end{abstract}

\\


\section{Introduction} \label{sec_introduction}
This work is concerned with model-based methods for derivative-free optimization (DFO).
Such methods are widely studied and have proven to be popular in practice \cite{Conn2009,Audet2017,Larson2019}.
Most commonly, model-based methods construct a local interpolant for the objective from selected evaluations which is then used in place of Taylor series in a standard framework, most commonly  trust-region methods \cite{Conn2007}.
The convergence and worst-case complexity of model-based DFO in the unconstrained setting is well-established.
Here, we are interested in nonconvex minimization over a convex set,
\begin{align}
	\min_{\bx\in C} f(\bx), \label{eq_main_problem}
\end{align}
where $f:\R^n\to\R$ is smooth (but nonconvex), and $C\subseteq\R^n$ is a closed, convex set with nonempty interior.
Such problems appear in the derivative-free setting, for example bound and linear constraints appearing in problems from climate science \cite{Tett2022,Scheurer2015}, deep learning \cite{Ughi2020} and systems design engineering \cite{Kulshreshtha2015}.
Although a model-based DFO framework for solving \eqref{eq_main_problem} was recently developed in \cite{LR_ConvexDFO2021}, it only considered the case of constructing linear interpolants from function evaluations (at feasible points).
Outside of structured problems such as nonlinear least-squares \cite{LR_DFOGN2019}, quadratic interpolants are the most common models used in model-based DFO \cite{Powell2006,Powell2009,Conn2009}.
In this work, we develop a comprehensive approximation theory for linear regression and underdetermined quadratic interpolation models based on sampling only from feasible points.
This theory fits the requirements of the method from \cite{LR_ConvexDFO2021}, and so allows this framework to be used with more practical model choices.

\subsection{Existing work}
Model-based DFO methods for unconstrained problems have well-established convergence \cite{Conn2007,Conn2009} and worst-case complexity theory \cite{Garmanjani2016,Larson2019}.
This theory relies on building (typically linear or quadratic) interpolants for the objective function based on samples of the objective at specific points, which must be ``fully linear'', i.e.~satisfying specific error bounds.
Unlike much approximation theory in numerical analysis (e.g.~\cite{Trefethen2020}), it is typically assumed that we have limited ability to select the sample points and instead should use as many pre-existing evaluations as possible (motivated by problems with computationally expensive objectives, an important DFO use case).
In this setting, the model accuracy can be bounded by quantities only depending on the choice of sampling points (and information about the smoothness of the objective).
Originally in \cite{Conn2007}, the concept of ``$\Lambda$-poised'' interpolation sets was introduced as a sufficient condition for constructing fully linear interpolation models, as well as developing procedures for improving the choice of sample points (while retaining as many existing points as possible) to construct a $\Lambda$-poised set.
This was extended in \cite{Conn2008} for polynomial regression and underdetermined polynomial interpolation models.
A specific form of underdetermined quadratic interpolation, namely minimum Frobenius interpolation, was introduced and extensively studied in \cite{Powell2004,Powell2004a} and is the foundation of several state-of-the-art codes \cite{Powell2006,Powell2009,LR_DFOLS2019}.
An analysis of minimum Frobenius interpolation in terms of $\Lambda$-poisedness/fully linear error bounds was given in \cite{Conn2009}.
Alternate model constructions have been proposed based on weighted linear regression \cite{Billups2013}, underdetermined quadratic interpolation based on Sobolev norms \cite{Zhang2014}, radial basis functions (RBF) \cite{Wild2008} and combinations of polynomial and RBFs \cite{Augustin2017}, for example.

Many model-based DFO methods have been proposed for constrained optimization; see \cite[Section 7]{Larson2019} for a thorough survey.
We note in particular the following strictly feasible methods: \cite{Conejo2013} studies convex-constrained problems, \cite{Powell2009,Gratton2011} and \cite[Section 6.3]{Wild2009} consider bound constraints, and \cite{Gumma2014} considers linear inequality constraints.
Strictly feasible methods for nonconvex-constrainted problems have also been considered in \cite{Conn1998,Conejo2015,Xuan2024}.
Of these, where convergence was established it was based on the assumption that fully linear models could be constructed using the same $\Lambda$-poised interpolation sets as in the unconstrained case.
However, this is potentially impractical, as in some settings ``\emph{it may be impossible to obtain a fully linear model using only feasible points}'' \cite[p.~362]{Larson2019}.

More recently, \cite{LR_ConvexDFO2021} introduces a model-based DFO method for solving \eqref{eq_main_problem} by adapting the unconstrained approach from \cite{Conn2009a} based on gradient-based trust-region methods for convex-constrained optimization \cite[Chapter 12]{Conn2000}.
The method has convergence and worst-case complexity results which align with unconstrained model-based DFO \cite{Garmanjani2016} and convex-constrained gradient-based trust-region methods \cite{Cartis2012b}.
To make this approach practical, it introduced a weaker notion of fully linear models which are sufficient for algorithm convergence, but which can be attained using linear interpolation over a $\Lambda$-poised set of only feasible points.
It also demonstrated how to modify an interpolation set to make it $\Lambda$-poised (without changing too many points).
However, linear models are not typically used in practice, except for special cases (such as nonlinear least-squares minimization \cite{LR_DFOGN2019}), with quadratic models being preferred \cite{Powell2006,Powell2009,LR_DFOLS2019}.

We additionally note that the BOBYQA code \cite{Powell2009} for bound-constrained optimization uses maximization of Lagrange polynomials inside the feasible region---similar to how we will define $\Lambda$-poisedness in this setting---but this was a heuristic with no theoretical analysis.\footnote{In fact, our results here give some theoretical justification to this approach.}

\subsection{Contributions}
The contributions of this work are:
\begin{itemize}
    \item A generalization of the linear interpolation theory (based on sampling feasible points) from \cite{LR_ConvexDFO2021} to the case of linear regression models. This includes fully linear error bounds and procedures to generate $\Lambda$-poised interpolation sets (which are simpler than those developed in \cite{LR_ConvexDFO2021}). 
    \item A full approximation theory for underdetermined quadratic interpolation (in the style of \cite{Powell2004,Powell2004a}) based on sampling only feasible points. This includes derivation of fully linear error bounds, defining the relevant notion of $\Lambda$-poisedness in this setting, and procedures for efficiently generating $\Lambda$-poised sets. By construction, in the case of a full interpolation set of $(n+1)(n+2)/2$ points for quadratic interpolation, this approach produces the same result as standard quadratic interpolation, and so our results also apply to this setting.
\end{itemize}
While the linear regression theory is of independent interest for DFO applied to noisy objectives \cite{Larson2016}, it is primarily included as a prerequisite for deriving the quadratic results.
The results proven in both cases are sufficient to enable these model constructions to be used in the algorithmic framework from \cite{LR_ConvexDFO2021} (with the associated convergence and worst-case complexity guarantees).

While the formulations and results are very similar to those for the unconstrained case from \cite{Conn2008}, the constrained setting necessitates new approaches for proving the results (which when applied to the $C=\R^n$ case provide simpler proofs of existing results).

\paragraph{Organization}
A summary of the modified fully linear model accuracy requirement and associated algorithm from \cite{LR_ConvexDFO2021} is given in \secref{sec_background}.
That $\Lambda$-poisedness (only over the feasible region) guarantees fully linear interpolation models is first shown for linear regression models in \secref{sec_linreg} and then for (underdetermined) quadratic interpolation models in \secref{sec_quad_interp}.
Lastly, \secref{sec_constructing_models} provides concrete procedures for identifying and constructing $\Lambda$-poised quadratic interpolation models using only feasible points.

\paragraph{Notation}
We let $\|\cdot\|$ be the Euclidean norm of vectors and the operator 2-norm of matrices, and use $B(\bx,\Delta)$ for $\bx\in\R^n$ and $\Delta>0$ for the closed ball $\{\by\in\R^n : \|\by-\bx\|\leq\Delta\}$.

\section{Background} \label{sec_background}
This section summarizes the key background and results from \cite{LR_ConvexDFO2021}, which provide a model-based DFO algorithm for solving \eqref{eq_main_problem}.

We now outline how problem \eqref{eq_main_problem} can be solved using model-based DFO methods.
We begin with outlining our problem assumptions.

\begin{assumption} \label{ass_smoothness}
	The objective function $f:\R^n\to\R$ is bounded below by $\flow$ and continuously differentiable, and $\grad f$ is $L_{\grad f}$-Lipschitz continuous.
\end{assumption}

Although $\grad f$ exists, we assume that it is not available to the algorithm as an oracle.

\begin{assumption} \label{ass_feasible_set}
	The feasible set $C\subseteq\R^n$ is closed and convex with nonempty interior.
\end{assumption}

\assref{ass_feasible_set} implies that the Euclidean projection operator for $C$,
\begin{align}
	\proj_C(\by) := \argmin_{\bx\in C} \|\by-\bx\|^2,
\end{align}
is a well-defined function $\proj_C:\R^n\to C$ \cite[Theorem 6.25]{Beck2017}.
Our framework is based on a setting in which $\proj_C$ is inexpensive to evaluate (at least compared to evaluations of the objective $f$).
Several examples of feasible sets $C$ which have simple analytic expressions for $\proj_C$ are given in \cite[Table 6.1]{Beck2017}.

In this work, as in \cite{LR_ConvexDFO2021}, we aim to find a first-order optimal solution to \eqref{eq_main_problem}.
To measure this, we use the following first-order criticality measure from \cite[Section 12.1.4]{Conn2000}, defined for all $\bx\in C$:
\begin{align}
	\pi^f(\bx) \defeq \left|\min_{\substack{\bx+\bd \in C\\ \|\bd\|\leq 1}} \grad f(\bx)^T \bd\right|. \label{eq_opt_measure}
\end{align}
From \cite[Theorem 12.1.6]{Conn2000}, we have that $\pi^f$ is a non-negative, continuous function of $\bx$ and $\pi^f(\bx^*)=0$ if and only if $\bx^*$ is a first-order critical point for \eqref{eq_main_problem}.
In the unconstrained case $C=\R^n$, then \eqref{eq_opt_measure} simplifies to $\pi^f(\bx) = \|\grad f(\bx)\|$.
For notational convenience, we define $\pi^f_k := \pi^f(\bx_k)$, where $\bx_k\in\R^n$ is the $k$-th iterate of our algorithm.

\subsection{Model Construction \& Accuracy}
We will consider a model-based DFO approach for solving \eqref{eq_main_problem}.
Specifically, at each iteration $k$ we construct a local quadratic approximation to $f$ which we hope to be accurate near the current iterate $\bx_k\in\R^n$:
\begin{align}
	f(\by) \approx m_k(\by) := c_k + \bg_k^T (\by-\bx_k) + \frac{1}{2}(\by-\bx_k)^T H_k (\by-\bx_k), \label{eq_model}
\end{align}
for some $c_k\in\R$, $\bg_k\in\R^n$ and $H_k\in\R^{n\times n}$ with $H_k=H_k^T$.

\begin{assumption} \label{ass_boundedhess}
	There exists $\kappa_H \geq 1$ such that $\|H_k\| \leq \kappa_H - 1$ for all $k$.
\end{assumption}

Convergence of our algorithm will require that the local models $m_k$ \eqref{eq_model} are sufficiently accurate approximations of the objective $f$ (at least on some iterations).
The required notion of `sufficiently accurate' is the following:

\begin{definition}  \label{def_fl}
	The model $m_k$ \eqref{eq_model} is \emph{$C$-fully linear} in $B(\bx_k,\Delta_k)$ if there exist $\kappaef,\kappaeg>0$, independent of $k$, such that
	\begin{subequations} \label{eq_fl}
	\begin{align}
		\max_{\substack{\bx_k+\bd \in C\\ \|\bd\|\leq \Delta_k}} |f(\bx_k+\bd) - m_k(\bx_k+\bd)| &\leq \kappaef \Delta_k^2, \label{eq_fl_f} \\
		\max_{\substack{\bx_k+\bd \in C\\ \|\bd\|\leq 1}} \left|(\grad f(\bx_k)-\bg_k)^T \bd\right| &\leq \kappaeg \Delta_k. \label{eq_fl_g}
	\end{align}
	\end{subequations}
	When we wish to refer to this notion in the abstract sense, we use the term ``$C$-full linearity''.
\end{definition}

We note in particular that \eqref{eq_fl_f} uses the constraint $\|\bd\|\leq \Delta_k$ and \eqref{eq_fl_g} uses $\|\bd\|\leq 1$.
This essentially corresponds to the standard definition of fully linear models \cite[Definition 3.1]{Conn2009a} in the unconstrained case $C=\R^n$.

Our algorithm will require the existence of two procedures related to $C$-full linearity: given a model $m_k$, iterate $\bx_k$ and trust-region radius $\Delta_k$, we assume that we can
\begin{itemize}
	\item Verify whether or not $m_k$ is $C$-fully linear in $B(\bx_k,\Delta_k)$; and
	\item If $m_k$ is not $C$-fully linear, create a new model (typically a small modification of $m_k$) which is $C$-fully linear in $B(\bx_k,\Delta_k)$.
\end{itemize}

Lastly, our model $m_k$ \eqref{eq_model} induces an approximate first-order criticality measure, namely
\begin{align}
 \pi^m(\bx) \defeq \left|\min_{\substack{\bx+\bd \in C\\ \|\bd\|\leq 1}} \bg_k^T \bd\right|. \label{eq_pik}
\end{align}
As above, for convenience we will define $\pi^m_k := \pi^m(\bx_k)$.

\subsection{Algorithm Specification}
Our DFO method is based on a trust-region framework.
That is, given the local model $m_k$ \eqref{eq_model}, we calculate a potential new iterate as $\bx_k+\bs_k$, where $\bs_k$ is an approximate minimizer of the trust-region subproblem
\begin{align}
 \min_{\substack{\bx+\bs \in C\\ \|\bs\|\leq \Delta_k}} m_k(\bx_k+\bs), \label{eq_trs}
\end{align}
where $\Delta_k>0$ is a parameter updated dynamically inside the algorithm.
Formally, we require the following, which can be achieved using a Goldstein-type linesearch method \cite[Algorithm 12.2.2 \& Theorem 12.2.2]{Conn2000}.

\begin{assumption} \label{ass_cdec}
	There exists a constant $c_1\in(0,1)$ such that the computed step $\bs_k$ satisfies $\bx_k+\bs_k\in C$, $\|\bs_k\|\leq\Delta_k$ and the generalized Cauchy decrease condition:
	\begin{align}
		m_k(\bx_k) - m_k(\bx_k + \bs_k) \geq c_1\pi^m_k \min\left(\frac{\pi^m_k}{1 + \|H_k\|}, \Delta_k, 1\right). \label{eq_2cdec}
	\end{align}
\end{assumption}

The full derivative-free algorithm for solving \eqref{eq_main_problem} from \cite{LR_ConvexDFO2021}, called CDFO-TR, is given in \algref{alg_cdfotr}.
The overall structure is similar to other model-based DFO methods, such as \cite[Algorithm 4.1]{Conn2009a} for unconstrained minimization.

\begin{algorithm}[tb]
\begin{algorithmic}[1]
\Require Starting point $\bx_0\in\R^n$ and trust-region radius $\Delta_0>0$.
\vspace{0.2em}
\Statex \underline{Parameters:} maximum trust-region radius $\Delta_{\max} \geq \Delta_0$, scaling factors $0 < \gammadec < 1 < \gammainc$, criticality constants $\epsilon_C,\mu>0$, and acceptance threshold $\eta\in(0,1)$.
\State Build an initial model $m_0$ \eqref{eq_model}.
\For{$k=0,1,2,\ldots$}
	\If{$\pi^m_k<\epsilon_C$ \textbf{and} \textit{($\pi^m_k<\mu^{-1}\Delta_k$ \textnormal{\textbf{or}} $m_k$ is not $C$-fully linear in $B(\bx_k,\Delta_k)$)}} \label{ln_main_start}
		\State \underline{Criticality step:} Set $\bx_{k+1}=\bx_k$. If $m_k$ is $C$-fully linear in $B(\bx_k,\Delta_k)$, set $\Delta_{k+1}=\gammadec\Delta_k$, otherwise set $\Delta_{k+1}=\Delta_k$. Construct $m_{k+1}$ to be $C$-fully linear in $B(\bx_{k+1},\Delta_{k+1})$.
	\Else \quad $\leftarrow$ \textit{$\pi^m_k\geq\epsilon_C$ or ($\pi^m_k\geq\mu^{-1}\Delta_k$ and $m_k$ is $C$-fully linear in $B(\bx_k,\Delta_k)$)}
		\State Approximately solve \eqref{eq_trs} to get a step $\bs_k$.
		\State Evaluate $f(\bx_k+\bs_k)$ and calculate ratio 
		\begin{align}
			\rho_k \defeq \frac{f(\bx_k)-f(\bx_k+\bs_k)}{m_k(\bx_k)-m_k(\bx_k+\bs_k)}. \label{eq_ratio_generic}
		\end{align}
		\If{$\rho_k \geq \eta$}
			\State \underline{Successful step:} Set $\bx_{k+1}=\bx_k+\bs_k$ and $\Delta_{k+1}=\min(\gammainc\Delta_k,\Delta_{\max})$. Form $m_{k+1}$ in any manner.
		\ElsIf{$m_k$ is not $C$-fully linear in $B(\bx_k,\Delta_k$)}
			\State \underline{Model-improving step:} Set $\bx_{k+1}=\bx_k$ and $\Delta_{k+1}=\Delta_k$, and construct $m_{k+1}$ to be $C$-fully linear in $B(\bx_{k+1},\Delta_{k+1})$.
		\Else \quad $\leftarrow$ \textit{$\rho_k<\eta$ and $m_k$ is $C$-fully linear in $B(\bx_k,\Delta_k)$}
			\State \underline{Unsuccessful step:} Set $\bx_{k+1}=\bx_k$ and $\Delta_{k+1}=\gammadec\Delta_k$. Form $m_{k+1}$ in any manner.
		\EndIf
	\EndIf \label{ln_main_end}
\EndFor
\end{algorithmic}
\caption{CDFO-TR: model-based DFO method for \eqref{eq_main_problem}, from \cite{LR_ConvexDFO2021}.}
\label{alg_cdfotr}
\end{algorithm}

We have the following global convergence and worst-case complexity results for \algref{alg_cdfotr} from \cite{LR_ConvexDFO2021}.

\begin{theorem}[Theorem 3.10 \& Corollary 3.15, \cite{LR_ConvexDFO2021}] \label{thm_limpif}
	If Assumptions~\ref{ass_smoothness}, \ref{ass_boundedhess} and \ref{ass_cdec} hold, then $\lim_{k\to\infty} \pi^f_k = 0$.
	Moreover, if $\epsilon\in(0,1]$ and $\epsilon_C \geq c_2 \epsilon$ for some constant $c_2>0$, then the number of iterations $k$ before \algref{alg_cdfotr} produces an iterate with $\pi^f_k < \epsilon$ is at most $\bigO(\kappa_H \kappa_d^2 \epsilon^{-2})$, where $\kappa_d := \max(\kappaef, \kappaeg)$.
\end{theorem}

In \cite{LR_ConvexDFO2021}, details are given on how to construct fully linear models based on linear interpolation to feasible points.
Although linear models can be practical for some structured problems, such as nonlinear least-squares objectives (e.g.~\cite{LR_DFOGN2019}), quadratic models are generally preferred.
The remainder of this paper is devoted to the construction of fully linear quadratic models by only sampling the objective at feasible points.

\section{Linear Regression Models} \label{sec_linreg}
We first consider how to extend the linear interpolation approximation theory from \cite{LR_ConvexDFO2021} to the case of linear regression models (with the slight extra generalization that the base point $\bx$ does not need to be an interpolation point).
The purpose of this is twofold: regression models can be useful for modelling noisy objectives (e.g.~\cite{Billups2013}), and this theory will be necessary to develop the corresponding quadratic interpolation theory in \secref{sec_quad_interp}.
In the unconstrained case $C=\R^n$, these results were originally proven in \cite{Conn2008}.

Suppose we have sampled the function $f$ at $p$ points $\{\by_1,\ldots,\by_p\}\subset\R^n$ (where $p\geq n+1$), and given this information we wish to find a linear model
\begin{align}
	f(\by) \approx m(\by) := c + \bg^T (\by-\bx), \label{eq_lin_reg_model}
\end{align}
by solving
\begin{align}
	\min_{c,\bg\in\R\times\R^n} \: \sum_{t=1}^{p} (f(\by_t)-m(\by_t))^2.
\end{align}
Equivalently, we can find the $c$ and $\bg$ for our model by finding the least-squares solution to the $p\times(n+1)$ system
\begin{align}
	M \bmat{ c \\ \bg} := \bmat{1 & (\by_1-\bx)^T \\ \vdots & \vdots \\ 1 & (\by_p-\bx)^T} \bmat{ c \\ \bg} = \bmat{f(\by_1) \\ \vdots \\ f(\by_{p})},  \label{eq_linreg2}
\end{align}
which may be written as 
\begin{align}
	\bmat{ c \\ \bg} = M^{\dagger} \bmat{f(\by_1) \\ \vdots \\ f(\by_p)},
\end{align}
where $M^{\dagger}\in\R^{(n+1)\times p}$ is the Moore-Penrose pseudoinverse of $M$.

Later we will require the following standard properties of $M^{\dagger}$ (see, e.g.~\cite[Section 5.5.4]{Golub1996}).

\begin{lemma} \label{lem_pseudoinverse}
For $p\geq n+1$, the Moore-Penrose pesudoinverse $M^{\dagger}$ of $M$ \eqref{eq_linreg2} satisfies $(M^T)^{\dagger} = (M^{\dagger})^T$ and $M^{\dagger} M M^{\dagger} = M^{\dagger}$. The minimal-norm solution to the underdetermined system $M^T \bu=\bv$ is $\bu=(M^T)^{\dagger} \bv$.
\end{lemma}

The quality of the choice of interpolation points will be assessed by considering the associated set of Lagrange polynomials.
In this case, the Lagrange polynomials associated with our interpolation set are the linear functions
\begin{align}
	\ell_t(\by) := c_t + \bg_t^T (\by-\bx), \qquad \forall t=1,\ldots,p,
\end{align}
each defined by the least-squares regression problem
\begin{align}
	\min_{c_t,\bg_t\in\R\times\R^n} \: \sum_{s=1}^{p+1} (\delta_{s,t} - \ell_t(\by_s))^2, \qquad \forall t=1,\ldots,p,
\end{align}
or equivalently
\begin{align}
	\bmat{ c_t \\ \bg_t} = M^{\dagger} \bee_t, \qquad \forall t=1,\ldots,p.
\end{align}
Our notion of the sampled points $\{\by_1,\ldots,\by_{p}\}$ being a `good' choice is given by the Lagrange polynomials having small magnitude in the region of interest.
This is formalized in the following notion, which generalizes \cite[Definition 2.7]{Conn2008} to the convex feasible region $C$.

\begin{definition} \label{def_new_poised_reg}
Given $\Lambda\geq 1$, the set $\{\by_1,\ldots,\by_{p}\} \subset C$ is $\Lambda$-poised for linear regression in $B(\bx,\Delta)\cap C$ if $\{\by_2-\by_1, \ldots, \by_{p}-\by_1\}$ spans $\R^n$ and 
\begin{align}
	\max_{t=1,\ldots,p} |\ell_t(\by)| \leq \Lambda, \qquad \forall \by\in C \cap B(\bx,\min(\Delta,1)). \label{eq_new_poised_reg}
\end{align}
\end{definition}

We note that if $\{\by_1,\ldots,\by_{p}\}$ is $\Lambda$-poised for linear regression, then $M$ has full column rank (since $\{\by_2-\by_1, \ldots, \by_{p}-\by_1\}$ is assumed to span $\R^n$). 
This requirement also necessitates that  $p\geq n+1$. 

Assuming the $\Lambda$-poisedness of $\{\by_1,\ldots,\by_{p}\}$ will be sufficient for the associated linear regression model to be $C$-fully linear.
The proof of this follows a similar structure to \cite[Lemma 4.3 \& Theorem 4.4]{LR_ConvexDFO2021}, but with an increased complexity to the proofs coming from using regression instead of interpolation.

\begin{lemma} \label{lem_fl_reg}
	Suppose $f$ satisfies \assref{ass_smoothness} and $C$ satisfies \assref{ass_feasible_set}.
	Then if $\{\by_1,\ldots,\by_{p}\}$ is $\Lambda$-poised for linear regression in $B(\bx,\Delta)\cap C$ and $\|\by_t-\bx\| \leq \beta \min(\Delta, 1)$ for all $t=1,\ldots,p$ and some $\beta>0$, we have
	\begin{align}
		|m(\by) - f(\bx) - \grad f(\bx)^T (\by-\bx)| \leq \frac{p \Lambda L_{\grad f} \beta^2}{2} \min(\Delta,1)^2, \label{eq_linreg2_tmp2}
	\end{align}
	for all $\by\in B(\bx, \min(\Delta,1))\cap C$.
	If we also have $\bx\in C$, this in turn implies error bounds on $c$ and $\bg$ individually, namely
	\begin{align}
		|c - f(\bx)| \leq \frac{p \Lambda L_{\grad f} \beta^2}{2} \min(\Delta,1)^2, \label{eq_fl_intermediate_v2a}
	\end{align}
	and
	\begin{align}
		|(\by-\bx)^T (\bg-\grad f(\bx))| \leq p \Lambda L_{\grad f} \beta^2 \min(\Delta,1)^2, \label{eq_fl_intermediate_v2}
	\end{align}
	for all $\by\in B(\bx, \min(\Delta,1))\cap C$.
\end{lemma}
\begin{proof}
	We begin by defining the residual of the least-squares problem \eqref{eq_linreg2} as
	\begin{align}
		\br := M \bmat{ c \\ \bg} - \bmat{f(\by_1) \\ \vdots \\ f(\by_{p})} = (M M^{\dagger} - I) \bmat{f(\by_1) \\ \vdots \\ f(\by_{p})} \in \R^{p}.
	\end{align}
	Then for all $t=1,\ldots,p$ have
	\begin{align}
		m(\by_t) - f(\bx) - \grad f(\bx)^T (\by_t-\bx) &= \bmat{c \\ \bg}^T \underbrace{\bmat{1 \\ \by_t-\bx}}_{=M^T\bee_t} - f(\bx) - \grad f(\bx)^T (\by_t-\bx), \\
		&= \br^T \bee_t + f(\by_t) - f(\bx) - \grad f(\bx)^T (\by_t-\bx).
	\end{align}
	
	Now, fix $\by\in B(\bx, \min(\Delta,1))\cap C$.
	Since $M$ has full column rank, its rows span $\R^{n+1}$ and so there exist constants $\alpha_t(\by)$ such that
	\begin{align}
		\bmat{1 \\ \by-\bx} = \sum_{t=1}^{p} \alpha_t(\by) \bmat{1 \\ \by_t-\bx} = M^T \balpha(\by). \label{eq_linreg2_alpha}
	\end{align}
	Of these, we take the $\balpha(\by)$ with minimal norm (c.f.~\lemref{lem_pseudoinverse}),
	\begin{align}
		\balpha(\by) = (M^T)^{\dagger} \bmat{1 \\ \by-\bx}. 
	\end{align}
	There are two key properties of $\balpha(\by)$: firstly,
	\begin{align}
		\ell_t(\by) = \bmat{c_t \\ \bg_t}^T \bmat{1 \\ \by-\bx} = \bee_t^T (M^{\dagger})^T \bmat{1 \\ \by-\bx} = \alpha_t(\by), \label{eq_linreg2_alpha2}
	\end{align}
	where the last equality follows  $(M^T)^{\dagger} = (M^{\dagger})^T$ (\lemref{lem_pseudoinverse}).
	Hence $|\alpha_t(\by)| = |\ell_t(\by)| \leq \Lambda$ from the $\Lambda$-poisedness condition.
	Secondly, we have
	\begin{align}
		\balpha(\by)^T \br = \bmat{1 \\ \by-\bx}^T M^{\dagger} (M M^{\dagger} - I) \bmat{f(\by_1) \\ \vdots \\ f(\by_{p})} = 0, \label{eq_linreg2_tmp1}
	\end{align}
	from $M^{\dagger} M M^{\dagger} = M^{\dagger}$ (\lemref{lem_pseudoinverse}).
	All together, we have
	\begin{align}
		|m(\by) - f(\bx) - \grad f(\bx)^T (\by-\bx)| &= \left|\bmat{c \\ \bg}^T \bmat{1 \\ \by-\bx} - \bmat{f(\bx) \\ \grad f(\bx)}^T \bmat{1 \\ \by-\bx}\right|, \\
		&= \left|\sum_{t=1}^{p} \alpha_t(\by) \left(\bmat{c \\ \bg}^T \bmat{1 \\ \by_t-\bx} - \bmat{f(\bx) \\ \grad f(\bx)}^T \bmat{1 \\ \by_t-\bx}\right)\right|, \\
		&= \left|\sum_{t=1}^{p} \alpha_t(\by) \left\{\br^T \bee_t + f(\by_t) - f(\bx) - \grad f(\bx)^T (\by_t-\bx)\right\}\right|, \\
		&\leq \left|\balpha(\by)^T \br\right| + \frac{L_{\grad f}}{2}\sum_{t=1}^{p} |\alpha_t(\by)| \cdot \|\by_t-\bx\|^2, \\
		&\leq \frac{p \Lambda L_{\grad f} \beta^2}{2} \min(\Delta,1)^2,
	\end{align}
	and we recover \eqref{eq_linreg2_tmp2}, where we used \eqref{eq_linreg2_tmp1}, $|\alpha_t(\by)| \leq \Lambda$ and $\|\by_t-\bx\| \leq \beta \min(\Delta,1)$ to get the last inequality.
 
	If $\bx\in C$, we can take  $\by=\bx$ in \eqref{eq_linreg2_tmp2} to get \eqref{eq_fl_intermediate_v2a}. 
	Combining \eqref{eq_fl_intermediate_v2a} with \eqref{eq_linreg2_tmp2} we get
	\begin{align}
		|(\by-\bx)^T (\bg-\grad f(\bx))| &\leq |c + \bg^T(\by-\bx) - f(\bx) - \grad f(\bx)^T (\by-\bx)| + |c - f(\bx)|, \\
		&\leq p \Lambda L_{\grad f} \beta^2 \min(\Delta,1)^2,
	\end{align}
	and we get \eqref{eq_fl_intermediate_v2}.
\end{proof}

\begin{theorem} \label{thm_fl_new_v2}
	Suppose the assumptions of \lemref{lem_fl_reg} hold and $\bx\in C$.
	Then the regression model \eqref{eq_lin_reg_model} is $C$-fully linear in $B(\bx,\Delta)$ with constants
	\begin{align}
		\kappaef = p \Lambda L_{\grad f} \beta^2 + \frac{L_{\grad f}}{2}, \qquad \text{and} \qquad \kappaeg = p \Lambda L_{\grad f} \beta^2,
	\end{align}
	in \eqref{eq_fl}.
\end{theorem}
\begin{proof}
	Fix $\by\in B(\bx,\Delta)\cap C$.
	We first derive $\kappaef$ by considering the cases $\Delta> 1$ and $\Delta\leq 1$ separately.
	
	If $\Delta> 1$, then $\hat{\by} = \bx + \frac{1}{\Delta}(\by-\bx) \in B(\bx,1) = B(\bx,\min(\Delta,1))$.
	Since $C$ is convex and $\bx,\by\in C$ we also have $\hat{\by}\in C$.
	Hence \eqref{eq_linreg2_tmp2} gives
	\begin{align}
		|c + \bg^T(\hat{\by}-\bx) - f(\bx) - \grad f(\bx)^T (\hat{\by}-\bx)| &\leq \frac{p \Lambda L_{\grad f} \beta^2}{2} \min(\Delta,1)^2,
	\end{align}
	and so
	\begin{align}
		\left|c + \frac{1}{\Delta}\bg^T(\by-\bx) - f(\bx) - \frac{1}{\Delta}\grad f(\bx)^T (\by-\bx)\right| &\leq \frac{p \Lambda L_{\grad f} \beta^2}{2} \min(\Delta,1)^2.
	\end{align}
	This gives us
	\begin{align}
		|f(\by) - m(\by)| &\leq |f(\by) - f(\bx) - \grad f(\bx)^T (\by-\bx)| + |m(\by) - f(\bx) - \grad f(\bx)^T (\by-\bx)|, \\
		&\leq \frac{L_{\grad f}}{2}\|\by-\bx\|^2 + \Delta \left|\frac{1}{\Delta}c + \frac{1}{\Delta}\bg^T(\by-\bx) - \frac{1}{\Delta}f(\bx) - \frac{1}{\Delta}\grad f(\bx)^T (\by-\bx)\right|, \\
		&\leq \frac{L_{\grad f}}{2}\|\by-\bx\|^2 + \Delta\left[\left|c + \frac{1}{\Delta}\bg^T(\by-\bx) - f(\bx) - \frac{1}{\Delta}\grad f(\bx)^T (\by-\bx)\right|\right. \nonumber \\
		&\qquad\qquad\qquad\qquad\qquad\qquad \left. + \left|\left(\frac{1}{\Delta}-1\right)(c - f(\bx))\right|\right], \\
		&\leq \frac{L_{\grad f}}{2}\Delta^2 + \Delta\left[\frac{p \Lambda L_{\grad f} \beta^2}{2} \min(\Delta,1)^2\right. \nonumber \\
		&\qquad\qquad\qquad\qquad\qquad\qquad \left. + \left(1-\frac{1}{\Delta}\right)\frac{p \Lambda L_{\grad f} \beta^2}{2} \min(\Delta,1)^2\right], \\
		&\leq \frac{L_{\grad f}}{2}\Delta^2 + p \Lambda L_{\grad f} \beta^2 \Delta \min(\Delta,1)^2, \\
		&= \frac{L_{\grad f}}{2}\Delta^2 + p \Lambda L_{\grad f} \beta^2 \Delta^2,
	\end{align}
	where we use \eqref{eq_fl_intermediate_v2a} to get the fourth inequality, and the last line follows from $\Delta > 1$.
	Instead if $\Delta\leq 1$, then $\by\in B(\bx,\min(\Delta,1))\cap C$ already, and so \eqref{eq_linreg2_tmp2} immediately gives
	\begin{align}
		|f(\by) - m(\by)| &= |f(\by) - f(\bx) - \grad f(\bx)^T (\by-\bx)| + \left|m(\by) - f(\bx) - \grad f(\bx)^T (\by-\bx)\right|, \\
		&\leq \frac{L_{\grad f}}{2}\Delta^2 + \frac{p \Lambda L_{\grad f} \beta^2}{2} \min(\Delta,1)^2, \\
		&\leq \frac{L_{\grad f}}{2}\Delta^2 + \frac{p \Lambda L_{\grad f} \beta^2}{2} \Delta^2.
	\end{align}
	Either way, we get the desired value of $\kappaef$.
	
	To get $\kappaeg$ we now fix an arbitrary $\t{\by}\in B(\bx,1)\cap C$ and again consider the cases $\Delta\geq 1$ and $\Delta<1$ separately.
	First, if $\Delta\geq 1$, then $\t{\by} \in B(\bx,\min(\Delta,1))\cap C$.
	From \eqref{eq_fl_intermediate_v2} we get
	\begin{align}
		|(\t{\by}-\bx)^T (\bg-\grad f(\bx))| \leq p \Lambda L_{\grad f} \beta^2 \min(\Delta,1)^2 \leq p \Lambda L_{\grad f} \beta^2 \Delta,
	\end{align}
	where the second inequality follows from $\Delta \geq 1$.
	Alternatively, if $\Delta<1$ then the convexity of $C$ implies that $\hat{\by}\defeq \bx+\Delta(\t{\by}-\bx)\in B(\bx,\Delta)\cap C = B(\bx,\min(\Delta,1))\cap C$.
	Again from \eqref{eq_fl_intermediate_v2} we get
	\begin{align}
		|(\t{\by}-\bx)^T (\bg-\grad f(\bx))| &= \frac{1}{\Delta}|(\hat{\by}-\bx)^T (\bg-\grad f(\bx))|, \\
		&\leq p \Lambda L_{\grad f} \beta^2 \Delta^{-1} \min(\Delta,1)^2, \\
		&= p \Lambda L_{\grad f} \beta^2 \Delta,
	\end{align}
	where the last line follows from $\Delta<1$.
	Again, either way we get the desired value of $\kappaeg$.
\end{proof}

\section{Underdetermined Quadratic Interpolation Models} \label{sec_quad_interp}
We now consider the case of forming $C$-fully linear quadratic interpolation models.
Our approach follows that of \cite{Powell2004} for the unconstrained case, although the fully linear error bounds for this approach were shown later in \cite{Conn2009}.
We note that this approach is different to the underdetermined quadratic interpolation used in \cite{Conn2008}.

Here, we aim to construct a quadratic model
\begin{align}
	f(\by) \approx m(\by) := c + \bg^T (\by-\bx) + \frac{1}{2}(\by-\bx)^T H (\by-\bx), \label{eq_quad_model}
\end{align}
where $c\in\R$, $\bg\in\R^n$ and $H\in\R^{n\times n}$ with $H=H^T$.
We assume that our interpolation set is $\{\by_1,\ldots,\by_p\}$ with $p\in\{n+2,\ldots,(n+1)(n+2)/2\}$.
The case $p=(n+1)(n+2)/2$ corresponds to full quadratic interpolation \cite[Section 4.2]{Conn2007}.
We exclude the case $p=n+1$ which, from \eqref{eq_min_frob} below, corresponds to linear interpolation and was analyzed (in the convex-constrained case) in \cite[Section 4]{LR_ConvexDFO2021}.

If $n+2 \leq p<(n+1)(n+2)/2$ then there are infinitely many models satisfying the interpolation conditions $f(\by_t)=m(\by_t)$ for all $t=1,\ldots,p$.
So, following \cite{Powell2004} we choose the model with minimum Frobenius norm Hessian by solving
\begin{subequations} \label{eq_min_frob}
\begin{align}
	\min_{c, \bg, H\in\R\times\R^n\times\R^{n\times n}} &\:\: \frac{1}{4} \|H\|_F^2, \\
	\text{s.t.} \quad & f(\by_t) = m(\by_t), \qquad \forall t=1,\ldots,p. \label{eq_quad_interp_constraints}
\end{align}
\end{subequations}
This is a convex quadratic program, and (as shown in \cite{Powell2004}), reduces to solving the $(p+n+1)\times(p+n+1)$ linear system
\begin{align}
    F \left[\begin{array}{c} \lambda_1 \\ \vdots \\ \lambda_p \\ \hline c \\ \bg \end{array}\right] := \left[\begin{array}{c|c} Q & M \\ \hline M^T & 0 \end{array}\right] \left[\begin{array}{c} \lambda_1 \\ \vdots \\ \lambda_p \\ \hline c \\ \bg \end{array}\right] = \left[\begin{array}{c} f(\by_1) \\ \vdots \\ f(\by_p) \\ \hline 0 \\ \b{0} \end{array}\right] \in \R^{p+n+1}, \label{eq_quad_interp_system}
\end{align}
where $M\in\R^{p\times(n+1)}$ is from the linear regression problem \eqref{eq_linreg2} and $Q\in\R^{p\times p}$ has entries $Q_{i,j} = \frac{1}{2}[(\by_i-\bx)^T(\by_j-\bx)]^2$ for $i,j=1,\ldots,p$.
The solution to \eqref{eq_quad_interp_system} immediately gives us $c$ and $\bg$; the (symmetric) model Hessian is given by $H = \sum_{t=1}^{p} \lambda_t (\by_t-\bx)(\by_t-\bx)^T$.
We note that $Q$ is symmetric positive semidefinite \cite[eq.~2.10]{Powell2004}, and $\lambda_1,\ldots,\lambda_p$ are the Lagrange multipliers associated with constraints \eqref{eq_quad_interp_constraints}.

The Lagrange polynomials associated with our interpolation set are
\begin{align}
	\ell_t(\by) := c_t + \bg_t^T(\by-\bx) + \frac{1}{2}(\by-\bx)^T H_t (\by-\bx), \qquad \forall t=1,\ldots,p,
\end{align}
where $c_t$, $\bg_t$ and $H_t$ come from solving \eqref{eq_min_frob} with \eqref{eq_quad_interp_constraints} replaced by $m(\by_s) = \delta_{s,t}$ for $s=1,\ldots,p$.
Equivalently, we have
\begin{align}
	F \bmat{\blambda_t \\ c_t \\ \bg_t} = \bee_t, \qquad \forall t=1,\ldots,p, \label{eq_quad_lagrange_system}
\end{align}
and $H_t = \sum_{s=1}^{p} [\blambda_t]_s (\by_s-\bx)(\by_s-\bx)^T$.
This gives
\begin{align}
	\ell_t(\by) &= c_t + \bg_t(\by-\bx) + \frac{1}{2}\sum_{s=1}^{p} [\blambda_t]_s\left[(\by-\bx)^T (\by_s-\bx)\right]^2, \\
	&= \bmat{\blambda_t \\ c_t \\ \bg_t}^T \underbrace{\bmat{\{\frac{1}{2}\left[(\by-\bx)^T (\by_s-\bx)\right]^2\}_{s=1,\ldots,p} \\ 1 \\ \by-\bx}}_{=: \bphi(\by)}, \\
	&= \bee_t^T F^{-1} \bphi(\by). \label{eq_min_frob_lagrange}
\end{align}
Given \eqref{eq_quad_interp_system} and \eqref{eq_quad_lagrange_system}, we conclude that 
\begin{align}
	\bmat{\lambda_1 \\ \vdots \\ \lambda_p} = \sum_{t=1}^{p} f(\by_t) \blambda_t, \qquad c = \sum_{t=1}^{p} f(\by_t) c_t, \quad \text{and} \quad \bg = \sum_{t=1}^{p} f(\by_t) \bg_t,
\end{align}
which gives
\begin{align}
	m(\by) = \sum_{t=1}^{p} f(\by_t) \ell_t(\by). \label{eq_mfn_frob_lin_comb}
\end{align}
We now have our new definition of $\Lambda$-poisedness:

\begin{definition} \label{def_new_poised_quad}
Given $\Lambda\geq 1$, the interpolation set $\{\by_1,\ldots,\by_p\} \subset C$ is $\Lambda$-poised for minimum Frobenius norm interpolation in $B(\bx,\Delta)\cap C$ if $F$ is invertible, and 
\begin{align}
	\max_{t=1,\ldots,p} |\ell_t(\by)| \leq \Lambda, \qquad \forall \by\in C \cap B(\bx,\min(\Delta,1)). \label{eq_new_poised_quad}
\end{align}
\end{definition}

\begin{lemma} \label{lem_quad_is_linreg}
	If the set $\{\by_1,\ldots,\by_p\} \subset C$ is $\Lambda$-poised for minimum Frobenius norm interpolation, then it is $\sqrt{p}\,\Lambda$-poised for linear regression.
\end{lemma}
\begin{proof}
        Since $F$ is invertible, the sub-matrix $M$ has full column rank by \cite[Theorem 3.3]{Benzi2005}, and so $\{\by_2-\by_1,\ldots,\by_p-\by_1\}$ spans $\R^n$.
        The remainder of this proof is based on the argument in \cite[p.~83]{Conn2009}.
	We note that since $H=0$ is a global minimizer of the objective function \eqref{eq_min_frob}, that if $f$ is linear then we have exact interpolation, $m=f$.
	Applying this to the functions $f(\by)=1$ and $f(\by)=(\by-\bx)^T \bee_i$ for $i=1,\ldots,n$, we get from \eqref{eq_mfn_frob_lin_comb} that
	\begin{align}
		1 = \sum_{t=1}^{p} \ell_t(\by), \qquad (\by-\bx)^T \bee_i = \sum_{t=1}^{p} (\by_t-\bx)^T \bee_i \cdot \ell_t(\by), \quad \forall i=1,\ldots,n.
	\end{align}
	Denoting $\bell(\by)\in\R^p$ as the vector of all $\ell_1(\by),\ldots,\ell_p(\by)$, these are equivalent to
	\begin{align}
		M^T \bell(\by) = \bmat{1 \\ \by-\bx}.
	\end{align}
	Hence $\bell(\by)$ is another solution to the (underdetermined) system \eqref{eq_linreg2_alpha}.
	Since the minimal norm solution to \eqref{eq_linreg2_alpha} was $\balpha(\by)$, we must have $\|\balpha(\by)\| \leq \|\bell(\by)\|$.
	However from \eqref{eq_linreg2_alpha2} we know $\balpha(\by) = \bell^{\text{reg}}(\by)$, where $\bell^{\text{reg}}(\by)$ are the Lagrange polynomials associated with linear regression for $\{\by_1,\ldots,\by_p\}$.
	Thus $\|\bell^{\text{reg}}(\by)\| \leq \|\bell(\by)\|$.
	Now, fixing any $\by \in B(\bx,\min(\Delta,1))\cap C$ we get
	\begin{align}
		\|\bell^{\text{reg}}(\by)\|_{\infty} \leq \|\bell^{\text{reg}}(\by)\| \leq \|\bell(\by)\| \leq \sqrt{p}\,\|\bell(\by)\|_{\infty} \leq \sqrt{p}\,\Lambda,
	\end{align}
	and we are done.
\end{proof}

We are now in a position to construct our fully linear error bounds.
We will begin by proving a bound on the size of the model Hessian, which requires the following technical result.

\begin{lemma} \label{lem_ineq_technical_general}
    Fix $t>1$ and $c_1,c_2\geq 0$.
    If $a,b\in\R$ satisfy $|a+b| \leq c_1$ and $|ta+b| \leq c_2$ then $|a| \leq (c_1+c_2)/(t-1)$ and $|b| \leq (tc_1+c_2)/(t-1)$.
\end{lemma}
\begin{proof}
    We first prove the bound on $|a|$.
    To find a contradiction, first suppose that $a > (c_1+c_2)/(t-1)$ and so $a>0$.
    Since $a+b \geq -c_1$ we have $b\geq -c_1-a$, which means $ta+b \geq (t-1)a-c_1 > c_2$, contradicting $|ta+b| \leq c_2$.
    Instead, suppose that $a < -(c_1+c_2)/(t-1)$ and so $a<0$.
    Then $a+b \leq c_1$ means $b \leq c_1-a$ and so $ta+b \leq (t-1)a + c_1 < -c_2$, again contradicting $|ta+b| \leq c_2$.

    The bound on $|b|$ follows from similar reasoning.
    First suppose that $b > (tc_1+c_2)/(t-1)$ and so $b>0$.
    Since $a+b \leq c_1$ we have $a\leq c_1-b$ and so $ta+b \leq tc_1 - (t-1)b < -c_2$, contradicting $|ta+b| \leq c_2$.
    Instead suppose that $b< -(tc_1+c_2)/(t-1)$ and so $b<0$.
    Then since $a+b \geq -c_1$ we have $a \geq -c_1-b$ and so $ta+b \geq -tc_1 - (t-1)b > c_2$, again contradicting $|ta+b| \leq c_2$.
\end{proof}

We can now give our bound on the model Hessian.
In the existing (unconstrained) theory, this is presented as a bound on $\|H\|$ \cite[Theorem 5.7]{Conn2009}, but here we only need to consider specific Rayleigh-type quotients.

\begin{lemma} \label{lem_hess_bd}
	Suppose $f$ and $C$ satisfy Assumptions~\ref{ass_smoothness} and \ref{ass_feasible_set} respectively.
	Then if $\bx\in C$, $\{\by_1,\ldots,\by_p\}$ is $\Lambda$-poised for underdetermined quadratic interpolation in $B(\bx,\Delta)\cap C$ and $\|\by_t-\bx\| \leq \beta \min(\Delta, 1)$ for all $t=1,\ldots,p$ and some $\beta>0$, then the model $m$ generated by \eqref{eq_min_frob} has Hessian $H$ satisfying
	\begin{align}
		\max_{s,t=1,\ldots,p} \frac{|(\by_{s}-\bx)^T H (\by_{t}-\bx)|}{\beta^2 \min(\Delta,1)^2} \leq \kappa_H := L_{\grad f} p \left[8\Lambda\beta^2 + 36\Lambda\beta + 58\Lambda + 6\right].
	\end{align}
\end{lemma}
\begin{proof}
    First, fix $u\in\{1,\ldots,p\}$ and consider the associated Lagrange polynomial
    \begin{align}
        \ell_u(\by) = c_u + \bg_u^T (\by-\bx) + \frac{1}{2}(\by-\bx)^T H_u (\by-\bx).
    \end{align}
    Additionally, fix $s,t\in\{1,\ldots,p\}$, and we will first provide a bound on 
    \begin{align}
        |(\by_s-\bx)^T H_u (\by_t-\bx)|.
    \end{align}
    To do this, we consider the value of $\ell_u$ at five different points: $\bx$, $\by_s$, $\by_t$, plus
    \begin{align}
        \hat{\by}_s := \bx + \frac{1}{\hat{\beta}}(\by_s-\bx), \qquad \text{and} \qquad \hat{\by}_t := \bx + \frac{1}{\hat{\beta}}(\by_t-\bx),
    \end{align}
    where $\hat{\beta}:=\max(\beta,2)$.
    Since $\|\by_s-\bx\| \leq \beta\min(\Delta,1)$ we have $\|\hat{\by}_s-\bx\| \leq \frac{\beta}{\hat{\beta}} \min(\Delta,1) \leq \min(\Delta,1)$, and $\hat{\by}_s\in C$ by convexity, since $\bx,\by_s\in C$.
    Similarly, we also have $\hat{\by}_t\in B(\bx,\min(\Delta,1))\cap C$.
    For these points, from the definition of $\Lambda$-poisedness we know $|\ell_u(\bx)|, |\ell_u(\hat{\by}_s)|, |\ell_u(\hat{\by}_t)| \leq \Lambda$ and by definition of Lagrange polynomials we have $\ell_u(\by_s),\ell_u(\by_t)\in\{0,1\}$ and so $|\ell_u(\by_s)|, |\ell_u(\by_t)| \leq 1$.

    From $|\ell_u(\bx)| \leq \Lambda$ we have $|c_u| \leq \Lambda$, and so $|\ell_u(\hat{\by}_s)| \leq \Lambda$ implies
    \begin{align}
	\Lambda &\geq \left|c_u + \bg_u^T(\hat{\by}_s-\bx) + \frac{1}{2}(\hat{\by}_s-\bx)^T H_u (\hat{\by}_s-\bx)\right|, \\
	&= \left|c_u + \frac{1}{\hat{\beta}}\bg_u^T (\by_s-\bx) + \frac{1}{2\hat{\beta}^2} (\by_s-\bx)^T H_u (\by_s-\bx)\right|, \\
	&\geq \frac{1}{\hat{\beta}^2}\left|\hat{\beta}\bg_u^T (\by_s-\bx) + \frac{1}{2} (\by_s-\bx)^T H_u (\by_s-\bx)\right| - |c_u|,
\end{align}
    where the last line follows from the reverse triangle inequality.
    Together with $|\ell_u(\by_s)| \leq 1$, we get
    \begin{align}
        \left|\bg_u^T (\by_s-\bx) + \frac{1}{2}(\by_s-\bx)^T H_u (\by_s-\bx)\right| &\leq 1 + |c_u| \leq \Lambda+1, \\
        \left|\hat{\beta}\bg_u^T (\by_s-\bx) + \frac{1}{2} (\by_s-\bx)^T H_u (\by_s-\bx)\right| &\leq 2\Lambda \hat{\beta}^2.
    \end{align}
    Since $\hat{\beta}>1$ by definition, applying \lemref{lem_ineq_technical_general} we conclude
    \begin{align}
        \left|\bg_u^T (\by_s-\bx)\right| &\leq \frac{\Lambda+1+2\Lambda\hat{\beta}^2}{\hat{\beta}-1}, \label{eq_hess_bdd_tmp1} \\
        \frac{1}{2}\left|(\by_s-\bx)^T H_u (\by_s-\bx)\right| &\leq \frac{\hat{\beta}(\Lambda+1)+2\Lambda\hat{\beta}}{\hat{\beta}-1}. \label{eq_hess_bdd_tmp2}
    \end{align}
    Since $s$ was arbitrary, the same inequalities hold with $\by_s$ replaced by  $\by_t$.

    Now, consider the point
    \begin{align}
        \hat{\by}_{s,t} := \bx + \frac{1}{2}(\hat{\by}_s-\bx) + \frac{1}{2}(\hat{\by}_t-\bx).
    \end{align}
    Since $\bx,\hat{\by}_s,\hat{\by}_t\in C$, we have $\hat{\by}_{s,t}\in C$, and also
    \begin{align}
        \|\hat{\by}_{s,t}-\bx\| \leq \frac{1}{2}\|\hat{\by}_s-\bx\| + \frac{1}{2}\|\hat{\by}_t-\bx\| \leq \min(\Delta,1),
    \end{align}
    and so $|\ell_u(\hat{\by}_{s,t})| \leq \Lambda$.
    Written in full, this is
    \begin{align}
        &\left|c_u + \frac{1}{2\hat{\beta}} \bg_u^T (\by_s-\bx) + \frac{1}{2\hat{\beta}} \bg_u^T (\by_t-\bx)\right. \nonumber \\
        &\qquad\qquad \left. + \frac{1}{2}\left(\frac{1}{2\hat{\beta}}(\by_s-\bx) + \frac{1}{2\hat{\beta}}(\by_t-\bx)\right)^T H_u \left(\frac{1}{2\hat{\beta}}(\by_s-\bx) + \frac{1}{2\hat{\beta}}(\by_t-\bx)\right)\right| \leq \Lambda.
    \end{align}
    That is,
    \begin{align}
        \frac{1}{4\hat{\beta}^2}|(\by_s-\bx)^T H_u (\by_t-\bx)| &\leq \Lambda + |c_u| + \frac{1}{2\hat{\beta}}\left(|\bg_u^T (\by_s-\bx)| + |\bg_u^T (\by_t-\bx)|\right) \nonumber \\
        &\qquad\qquad + \frac{1}{8\hat{\beta}^2} \left(|(\by_s-\bx)^T H_u (\by_s-\bx)| + |(\by_t-\bx)^T H_u (\by_t-\bx)|\right).
    \end{align}
    Applying $|c_u| \leq \Lambda$, \eqref{eq_hess_bdd_tmp1} and \eqref{eq_hess_bdd_tmp2}, we conclude
    \begin{align}
        \frac{1}{4\hat{\beta}^2}|(\by_s-\bx)^T H_u (\by_t-\bx)| &\leq 2\Lambda + \frac{\Lambda+1+2\Lambda\hat{\beta}^2}{\hat{\beta}(\hat{\beta}-1)} + \frac{\hat{\beta}(\Lambda+1)+2\Lambda\hat{\beta}}{2\hat{\beta}^2(\hat{\beta}-1)},
    \end{align}
    or
    \begin{align}
        |(\by_s-\bx)^T H_u (\by_t-\bx)| &\leq \tilde{\kappa},
    \end{align}
    for all $s,t,u=1,\ldots,p$, where
    \begin{align}
        \tilde{\kappa} := 8\Lambda \hat{\beta}^2 + \frac{4\Lambda\hat{\beta}+4\hat{\beta}+8\Lambda\hat{\beta}^3}{\hat{\beta}-1} + \frac{2\hat{\beta}(\Lambda+1)+4\Lambda\hat{\beta}}{\hat{\beta}-1} = 8\Lambda \hat{\beta}^2 + \frac{8\Lambda\hat{\beta}^3 + 10\Lambda\hat{\beta} + 6\hat{\beta}}{\hat{\beta}-1}.
    \end{align}
    More simply, we have
    \begin{align}
        \tilde{\kappa} &= 8\Lambda\hat{\beta}^2 + \frac{8\Lambda(\hat{\beta}-1)^3 + 24\Lambda(\hat{\beta}-1)^2 + (34\Lambda+6)(\hat{\beta}-1) + 18\Lambda+6}{\hat{\beta}-1},\\
        &\leq 8\Lambda (\beta+2)^2 + 8\Lambda (\beta+1)^2 + 24\Lambda(\beta+1) + 34\Lambda + 6 + 18\Lambda + 6, \\
        &= 16\Lambda\beta^2 + 72\Lambda\beta + 116\Lambda + 12, \label{eq_tilde_kappaH_bound}
    \end{align}
    using $2\leq \hat{\beta}=\max(\beta,2) \leq \beta+2$ to get the inequality.
    Now let us turn our attention to the model $m$ \eqref{eq_min_frob}.
    Note that we may add/subtract a linear function to $f(\by)$ without changing $H$ (since this just changes $c$ and $\bg$ in \eqref{eq_min_frob}).
    So, we consider the model $\t{m}$ generated by interpolation to $\t{f}(\by) := f(\by) - f(\bx) - \grad f(\bx)^T (\by-\bx)$.
    From \eqref{eq_mfn_frob_lin_comb} we may write
    \begin{align}
        H = \t{H} = \sum_{u=1}^{p} \t{f}(\by_u) H_u,
    \end{align}
    where $H_u$ are the Hessians of the Lagrange polynomials.
    From \assref{ass_smoothness}, we also have $|\t{f}(\by)| \leq \frac{L_{\grad f}}{2} \|\by-\bx\|^2$.
    Then for any $s,t=1,\ldots,p$ we conclude
    \begin{align}
	|(\by_s-\bx)^T H (\by_t-\bx)| &\leq \sum_{u=1}^{p} |\t{f}(\by_u)| \cdot |(\by_s-\bx)^T H_u (\by_t-\bx)|, \\
	&\leq \sum_{u=1}^{p} \frac{L_{\grad f}}{2} \tilde{\kappa} \|\by_u-\bx\|^2, \\
        &\leq \frac{L_{\grad f}}{2} \tilde{\kappa} p \beta^2 \min(\Delta,1)^2.
    \end{align}
    and we are done after applying \eqref{eq_tilde_kappaH_bound}.
\end{proof}

\begin{remark} \label{rem_kappaH_compare}
    In the unconstrained case, \cite[Theorem 5.7]{Conn2009} gives the bound
    \begin{align}
        \|H\| \leq 4(p+1) \sqrt{\frac{(n+1)(n+2)}{2}} \: L_{\grad f} \Lambda \max(1, \Delta_{\max}^2),
    \end{align}
    where $\Delta_{\max}$ is an upper bound for $\Delta$.
    This result assumes $\|\by_t-\bx\| \leq \Delta$, and so to match our assumptions it requires $\|\by_t-\bx\| \leq \beta\min(\Delta,1)$ where $\beta=\Delta_{\max}$.
    Thus we see that \lemref{lem_hess_bd} improves on \cite[Theorem 5.7]{Conn2009} by a factor of $\bigO(n)$.
\end{remark}

Next, we can prove an analogous result to \lemref{lem_fl_reg}.

\begin{lemma} \label{lem_fl_intermediate_quad_new}
    Suppose the assumptions of \lemref{lem_hess_bd} are satisfied.
    Then the model $m$ generated by \eqref{eq_min_frob} satisfies
    \begin{align}
        \left|c + \bg^T(\by-\bx) - f(\bx) - \grad f(\bx)^T(\by-\bx)\right| \leq \frac{1}{2} p^{3/2} \Lambda (L_{\grad f} + \kappa_H) \beta^2 \min(\Delta,1)^2, \label{eq_min_frob_fl_part1}
    \end{align}
    for all $\by\in B(\bx,\min(\Delta,1))\cap C$.
    Furthermore if $\bx\in C$, we have
    \begin{align}
        |c - f(\bx)| \leq \frac{1}{2} p^{3/2} \Lambda (L_{\grad f} + \kappa_H) \beta^2 \min(\Delta,1)^2, \label{eq_min_frob_fl_part2}
    \end{align}
    and
    \begin{align}
        |(\bg-\grad f(\bx))^T (\by-\bx)| \leq p^{3/2} \Lambda (L_{\grad f} + \kappa_H) \beta^2 \min(\Delta,1)^2, \label{eq_min_frob_fl_part3}
    \end{align}
    for all $\by\in B(\bx,\min(\Delta,1))\cap C$.
\end{lemma}
\begin{proof}
    Fix $\by\in B(\bx,\min(\Delta,1))\cap C$.
    From \lemref{lem_quad_is_linreg}, our interpolation set is $\sqrt{p}\,\Lambda$-poised for linear regression.
	Therefore there exist constants $\{\alpha_t(\by)\}_{t=1}^{p}$ such that
	\begin{align}
		\bmat{1 \\ \by-\bx} = \sum_{t=1}^{p} \alpha_t(\by) \bmat{1 \\ \by_t-\bx}, \label{eq_min_frob_span}
	\end{align}
	where $|\alpha_t(\by)| = |\ell_t^{\text{reg}}(\by)| \leq \sqrt{p}\,\Lambda$ (provided the minimal norm solution is taken as in \eqref{eq_linreg2_alpha}).
    Thus we have
    \begin{align}
        &\left|c + \bg^T(\by-\bx) - f(\bx) - \grad f(\bx)^T(\by-\bx)\right| \nonumber \\
        &\qquad\qquad = \left|\bmat{c - f(\bx) \\ \bg - \grad f(\bx)}^T \bmat{1 \\ \by-\bx}\right|, \\
        &\qquad\qquad \leq \sum_{t=1}^{p} \sqrt{p}\:\Lambda \left|\bmat{c-f(\bx) \\ \bg-\grad f(\bx)}^T \bmat{1 \\ \by_t-\bx}\right|, \\
        &\qquad\qquad = \sqrt{p}\:\Lambda \sum_{t=1}^{p} \left|m(\by_t) - f(\bx) - \grad f(\bx)^T (\by_t-\bx) - \frac{1}{2}(\by_t-\bx)^T H (\by_t-\bx)\right|, \\
        &\qquad\qquad \leq \sqrt{p}\:\Lambda \sum_{t=1}^{p} \left[\left|f(\by_t) - f(\bx) - \grad f(\bx)^T (\by_t-\bx)\right| + \frac{1}{2}\left|(\by_t-\bx)^T H (\by_t-\bx)\right|\right], \\
        &\qquad\qquad \leq p^{3/2} \Lambda \left(\frac{L_{\grad f}}{2} \beta^2 \min(\Delta,1)^2 + \frac{\kappa_H}{2} \beta^2 \min(\Delta,1)^2\right),
    \end{align}
    using \assref{ass_smoothness} and \lemref{lem_hess_bd} in the last line, and we have \eqref{eq_min_frob_fl_part1}.
    To get \eqref{eq_min_frob_fl_part2}, we just take $\by=\bx$ in \eqref{eq_min_frob_fl_part1}.
    Finally, to get \eqref{eq_min_frob_fl_part3}, we use 
    \begin{align}
        |(\bg-\grad f(\bx))^T (\by-\bx)| \leq |c + \bg^T (\by-\bx) - f(\bx) - \grad f(\bx)^T (\by-\bx)| + |c - f(\bx)|,
    \end{align}
    together with \eqref{eq_min_frob_fl_part1} and \eqref{eq_min_frob_fl_part2}.
\end{proof}

Finally, we can give our fully linear error bounds for underdetermined quadratic interpolation models.

\begin{theorem} \label{thm_fl_quad}
  Suppose $f$ and $C$ satisfy Assumptions~\ref{ass_smoothness} and \ref{ass_feasible_set} respectively.
  Then if $\bx\in C$, $\{\by_1,\ldots,\by_p\}$ is $\Lambda$-poised for underdetermined quadratic interpolation in $B(\bx,\Delta)\cap C$ and $\|\by_t-\bx\| \leq \beta \min(\Delta, 1)$ for all $t=1,\ldots,p$ and some $\beta>0$, then the model $m$ generated by \eqref{eq_min_frob} is fully linear in $B(\bx,\Delta)$ (in the sense of \defref{def_fl}) with constants
  \begin{align}
      \kappaef = \frac{1}{2}L_{\grad f} + \frac{3}{2} p^{3/2} \Lambda (L_{\grad f}+\kappa_H) \beta^2 + \frac{1}{2}p \Lambda^2 \kappa_H \beta^2, \quad \text{and} \quad \kappaeg = p^{3/2} \Lambda (L_{\grad f} + \kappa_H) \beta^2, \label{eq_fl_quad_constants}
  \end{align}
  in \eqref{eq_fl}, where $\kappa_H$ is defined in \lemref{lem_hess_bd}.
\end{theorem}
\begin{proof}
    First fix $\by\in B(\bx,\Delta)\cap C$ and we will derive $\kappaef$ by considering the cases $\Delta\geq 1$ and $\Delta<1$ separately.
    If $\Delta\geq 1$ then $\hat{\by} := \bx+\Delta^{-1}(\by-\bx)$ is in $B(\bx,1)=B(\bx,\min(\Delta,1))$ and $\hat{\by}\in C$ since $C$ is convex.
    We then apply \lemref{lem_fl_intermediate_quad_new} to get
    \begin{align}
        |(\by-\bx)^T (\bg-\grad f(\bx))| &= \Delta |(\hat{\by}-\bx)^T (\bg-\grad f(\bx))|, \\
        &\leq p^{3/2} \Lambda (L_{\grad f}+\kappa_H) \beta^2 \Delta \min(\Delta,1)^2, \\
        &\leq p^{3/2} \Lambda (L_{\grad f}+\kappa_H) \beta^2 \Delta^2, \label{eq_poised_quad_tmp1}
    \end{align}
    where the last inequality follows from $\min(\Delta,1)^2 = 1\leq \Delta$.
    Also, from \eqref{eq_min_frob_span} we have $\hat{\by}-\bx=\sum_{t=1}^{p} \alpha_t(\hat{\by}) (\by_t-\bx)$ with $|\alpha_t(\by)| \leq \sqrt{p}\,\Lambda$, and so from \lemref{lem_hess_bd},
    \begin{align}
        |(\by-\bx)^T H (\by-\bx)| &= \Delta^2 |(\hat{\by}-\bx)^T H (\hat{\by}-\bx)|, \\
        &\leq \Delta^2 \sum_{s,t=1}^{p} \left|\alpha_s(\hat{\by}) \alpha_t(\hat{\by}) (\by_s-\bx)^T H (\by_t-\bx)\right|, \\
        &\leq p \Lambda^2 \kappa_H \beta^2 \Delta^2 \min(\Delta,1)^2, \\
        &= p \Lambda^2 \kappa_H \beta^2 \Delta^2, \label{eq_poised_quad_tmp1a}
    \end{align}
    again using $\Delta\geq 1$ in the last line.
    
    Instead, if $\Delta<1$ then $\by\in B(\bx,\min(\Delta,1))\cap C$ and we apply \lemref{lem_fl_intermediate_quad_new} directly to get
    \begin{align}
        |(\by-\bx)^T (\bg-\grad f(\bx))| &\leq \frac{1}{2} p^{3/2} \Lambda (L_{\grad f}+\kappa_H) \beta^2 \min(\Delta,1)^2, \\
        &= p^{3/2} \Lambda (L_{\grad f}+\kappa_H) \beta^2 \Delta^2. \label{eq_poised_quad_tmp2}
    \end{align}
    Also, using \lemref{lem_hess_bd} we have
    \begin{align}
        |(\by-\bx)^T H (\by-\bx)| &\leq \sum_{t=1}^{p} \left|\alpha_s(\by) \alpha_t(\by) (\by_s-\bx)^T H (\by_t-\bx)\right|, \\
        &\leq p \Lambda^2 \kappa_H \beta^2 \min(\Delta,1)^2, \\
        &= p \Lambda^2 \kappa_H \beta^2 \Delta^2. \label{eq_poised_quad_tmp2a}
    \end{align}
    In either case, we use \eqref{eq_poised_quad_tmp1} and \eqref{eq_poised_quad_tmp1a}, or \eqref{eq_poised_quad_tmp2} and \eqref{eq_poised_quad_tmp2a}, with \assref{ass_smoothness} and \lemref{lem_fl_intermediate_quad_new} to get
    \begin{align}
        |f(\by)-m(\by)| &\leq |f(\by)-f(\bx) - \grad f(\bx)^T (\by-\bx)| \nonumber \\
        &\qquad + |c + \bg^T (\by-\bx) + \frac{1}{2}(\by-\bx)^T H (\by-\bx) - f(\bx) - \grad f(\bx)^T (\by-\bx)|, \\
        &\leq \frac{1}{2} L_{\grad f} \Delta^2 + |c-f(\bx)| + |(\by-\bx)^T (\bg-\grad f(\bx))| + \frac{1}{2}|(\by-\bx)^T H(\by-\bx)|, \\
        &\leq \frac{1}{2} L_{\grad f} \Delta^2 + \frac{1}{2} p^{3/2} \Lambda (L_{\grad f} + \kappa_H) \beta^2 \min(\Delta,1)^2  \nonumber \\
        &\qquad\qquad\qquad\qquad + p^{3/2} \Lambda (L_{\grad f}+\kappa_H) \beta^2 \Delta^2 + \frac{1}{2}p \Lambda^2 \kappa_H \beta^2 \Delta^2,
    \end{align}
    and we get the value of $\kappaef$ after $\min(\Delta,1)\leq \Delta$.

    To get $\kappaeg$ we now fix an arbitrary $\t{\by}\in B(\bx,1)\cap C$ and again consider the cases $\Delta\geq 1$ and $\Delta<1$ separately.
  First, if $\Delta\geq 1$, then $\t{\by} \in B(\bx,\min(\Delta,1))\cap C$, and applying \lemref{lem_fl_intermediate_quad_new} we get
  \begin{align}
      |(\t{\by}-\bx)^T (\bg-\grad f(\bx))| &\leq p^{3/2} \Lambda (L_{\grad f} + \kappa_H) \beta^2 \min(\Delta,1)^2, \\
      &\leq p^{3/2} \Lambda (L_{\grad f} + \kappa_H) \beta^2 \Delta, \label{eq_interp_quad_tmp6a}
  \end{align}
  since $\min(\Delta,1)^2 = 1 \leq \Delta$.
  Alternatively, if $\Delta<1$ then the convexity of $C$ implies that $\hat{\by}\defeq \bx+\Delta(\t{\by}-\bx)\in B(\bx,\Delta)\cap C = B(\bx,\min(\Delta,1))\cap C$.
  Again we apply \lemref{lem_fl_intermediate_quad_new} and get
  \begin{align}
      |(\t{\by}-\bx)^T (\bg-\grad f(\bx))| &= \Delta^{-1} |(\hat{\by}-\bx)^T (\bg-\grad f(\bx))|, \\
      &\leq p^{3/2} \Lambda (L_{\grad f} + \kappa_H) \beta^2 \Delta^{-1} \min(\Delta,1)^2, \\
      &= p^{3/2} \Lambda (L_{\grad f} + \kappa_H) \beta^2 \Delta. \label{eq_interp_quad_tmp6b}
  \end{align}
  The value for $\kappaeg$ then follows from \eqref{eq_interp_quad_tmp6a} and \eqref{eq_interp_quad_tmp6b}.
\end{proof}

\begin{remark}
	The fully linear error bounds for the unconstrained case \cite{Conn2008,Conn2009} give $\kappaef,\kappaeg = \bigO(p^{3/2} \Lambda (L_{\grad f} + \kappa_H))$, so \thmref{thm_fl_quad} matches the bound for $\kappaeg$ up to a factor of $\beta^2$ (but with a value of $\kappa_H$ which is $\bigO(n)$ smaller than the unconstrained case; see \remref{rem_kappaH_compare}).
    Our value of $\kappaef$ has an extra term of size $\bigO(p \Lambda^2 \kappa_H \beta^2)$ and so may be larger depending on the relative sizes of $p$ and $\Lambda$.
\end{remark}

\section{Constructing $\b{\Lambda}$-Poised Quadratic Models} \label{sec_constructing_models}

To meet all the requirements of \algref{alg_cdfotr}, we require procedures which can verify whether or not a given model is $C$-fully linear, and if not, modify it to be $C$-fully linear.
From \thmref{thm_fl_quad} it is clear that it suffices to verify/ensure that $\{\by_1,\ldots,\by_p\}$ is $\Lambda$-poised for underdetermined quadratic interpolation in $B(\bx,\Delta)\cap C$ and $\|\by_t-\bx\| \leq \beta \min(\Delta, 1)$ for all $t=0,\ldots,p$.

In the first case---verifying whether not the interpolation set is $\Lambda$-poised---we may follow the approach in \cite[Section 4.3.1]{LR_ConvexDFO2021} and simply maximize/minimize each $\ell_t$ in $B(\bx,\min(\Delta,1))\cap C$.
Checking $\|\by_t-\bx\| \leq \beta\min(\Delta,1)$ is straightforward.
Although maximizing/minimizing $\ell_t$, a possibly nonconvex quadratic function, in $B(\bx,\min(\Delta,1))\cap C$ may not be straightforward depending on the structure of $C$, we only need to know if the solution is above/below $\Lambda$ and not the exact solution.
Existing solvers, including potentially global optimization solvers, are sufficient for this.

We now consider the situation where $\{\by_1,\ldots,\by_p\}$ is not $\Lambda$-poised, and we wish to construct a new interpolation set which is $\Lambda$-poised.
There are two possible scenarios:
\begin{itemize}
	\item The associated matrix $F$ in \eqref{eq_quad_interp_system} is not invertible; or
	\item The matrix $F$ is invertible but $|\ell_t(\by)| > \Lambda$ for some $t=1,\ldots,p$ and $\by\in B(\bx,\min(\Delta,1))\cap C$.
\end{itemize}
We will describe how to handle both of these situations in \secref{sec_constructing_poised_sets}, but we will first need a technical result about how the matrix $F$ in \eqref{eq_quad_interp_system} changes as interpolation points are changed.

\subsection{Updating the matrix $\bm{F}$}
In this section we analyze changes in the determinant of the matrix $F$ as interpolation points are updated.
These results are based on ideas from \cite[Section 4]{Powell2004}, which considers changes in $F^{-1}$ from updating interpolation points, rather than $\det(F)$.

\begin{lemma} \label{lem_row_col_update}
	If $A$ is a symmetric, invertible matrix and we form $\t{A}$ by changing the $t$-th row and column of $A$ to $\t{\bv}$, then
	\begin{align}
		\det(\t{A}) = \left[(\bee_t^T A^{-1} \t{\bv})^2 + (\bee_t^T A^{-1} \bee_t) \t{\bv}^T (\bee_t - A^{-1} \t{\bv})\right] \det(A).
	\end{align}
\end{lemma}
\begin{proof}
	Denote the $t$-th row and column of $A$ by $\bv:= A\bee_t$. 
 We first note that $\t{A}$ can be written as a symmetric rank-3 update of $A$:
	\begin{align}
		\t{A} = A + (\t{\bv}-\bv)\bee_t^T + \bee_t (\t{\bv}-\bv)^T - (\t{v}_t-v_t) \bee_t \bee_t^T. \label{eq_low_rank_update}
	\end{align}
	The first two update terms add $\t{\bv}-\bv$ to the $t$-th column and row respectively, and the last term corrects for the double update of $A_{t,t}$ from the first two updates.
	We can then apply the matrix determinant lemma\footnote{That is, $\det(A+UV^T) = \det(I+V^T A^{-1} U) \det(A)$ if $A$ is invertible, e.g.~\cite[Corollary 18.1.4]{Harville2008}.} to get 
	\begin{align}
		\frac{\det(\t{A})}{\det(A)} &= \det\left(\begin{bmatrix} 1 + \bee_t^T A^{-1}(\t{\bv}-\bv) & \bee_t^T A^{-1} \bee_t & \bee_t^T A^{-1} \bee_t \\ (\t{\bv}-\bv)^T A^{-1} (\t{\bv}-\bv) & 1 + (\t{\bv}-\bv)^T A^{-1}\bee_t & (\t{\bv}-\bv)^T A^{-1}\bee_t \\ -(\t{v}_t-v_t)\bee_t^T A^{-1}(\t{\bv}-\bv) & -(\t{v}_t-v_t)\bee_t^T A^{-1}\bee_t & 1 - (\t{v}_t-v_t)\bee_t^T A^{-1}\bee_t \end{bmatrix}\right), \\
		&= \det\left(\begin{bmatrix} 1+\tau & \alpha & \alpha \\ \beta & 1+\tau & \tau \\ -\gamma\tau & -\gamma\alpha & 1-\gamma\alpha \end{bmatrix}\right),
	\end{align}
	where $\alpha := \bee_t^T A^{-1} \bee_t$, $\beta := (\t{\bv}-\bv)^T A^{-1} (\t{\bv}-\bv)$, $\gamma := \t{v}_t-v_t$ and $\tau := \bee_t^T A^{-1} (\t{\bv}-\bv) = (\t{\bv}-\bv)^T A^{-1} \bee_t$ since $A^{-1}$ is symmetric.
	Directly computing, we get
	\begin{align}
		\frac{\det(\t{A})}{\det(A)} &= (1+\tau)\left[(1+\tau)(1-\gamma\alpha) + \tau\gamma\alpha\right] - \beta\left[\alpha(1-\gamma\alpha) + \gamma\alpha^2\right] - \gamma\tau\left[\alpha\tau - \alpha(1+\tau)\right], \\
		&= (1+\tau)\left[1+\tau-\gamma\alpha\right] - \beta\alpha + \gamma\tau\alpha, \\
		&= (1+\tau)^2 - \alpha(\gamma+\beta). \label{eq_row_col_update_tmp1_alt}
	\end{align}
	Finally, since $\bv=A\bee_t$, we get
	\begin{align}
		\tau = \bee_t^T A^{-1} (\t{\bv}-\bv) = \bee_t^T A^{-1} \t{\bv} - \bee_t^T \underbrace{A^{-1} \bv}_{=\bee_t} = \bee_t^T A^{-1} \t{\bv}-1, \label{eq_row_col_update_tmp2_alt}
	\end{align}
	and
	\begin{align}
		\gamma+\beta &= (\t{\bv}-\bv)^T \bee_t + (\t{\bv}-\bv)^T A^{-1} (\t{\bv}-\bv), \\
		&= \t{\bv}^T \bee_t - \bv^T \bee_t + \t{\bv}^T A^{-1} \t{\bv} - 2\t{\bv}^T A^{-1} \bv + \bv^T A^{-1} \bv, \\
		&= \t{\bv}^T \bee_t - \bv^T \bee_t + \t{\bv}^T A^{-1} \t{\bv} - 2\t{\bv}^T \bee_t + \bv^T \bee_t, \\
		&= \t{\bv}^T (A^{-1} \t{\bv} - \bee_t). \label{eq_row_col_update_tmp3_alt}
	\end{align}
	Our result then follows from \eqref{eq_row_col_update_tmp1_alt} combined with the definition of $\alpha$, \eqref{eq_row_col_update_tmp2_alt} and \eqref{eq_row_col_update_tmp3_alt}.
\end{proof}

\begin{theorem} \label{thm_quad_update_det}
	Suppose the interpolation set $\{\by_1,\ldots,\by_p\}\in\R^n$ is such that the matrix $F$ in \eqref{eq_quad_interp_system} is invertible.
	If the point $\by_t$ for some $t\in\{1,\ldots,p\}$ is changed to a new point $\by\in\R^n$, then the new interpolation set yields a new matrix $\t{F}$ in \eqref{eq_quad_interp_system} which satisfies
	\begin{align}
		|\det(\t{F})| \geq \ell_t(\by)^2 |\det(F)|.
	\end{align}
\end{theorem}
\begin{proof}
	If $\ell_t(\by)=0$, then the result holds trivially, so we assume without loss of generality that $\ell_t(\by) \neq 0$.
	
	Changing the interpolation point $\by_t$ to $\by$ requires replacing the $t$-th row and column of $F$ by $\t{\bv}$, where
	\begin{subequations}
	\begin{align}
		\t{\bv}_s &= \frac{1}{2}\left[(\by-\bx)^T (\by_s-\bx)\right]^2, \qquad \text{if $s\in\{1,\ldots,p\}$ except for $s=t$}, \\
		\t{\bv}_t &= \frac{1}{2}\left[(\by-\bx)^T (\by-\bx)\right]^2, \\
		\t{\bv}_{p+2} &= 1, \\
		\t{\bv}_{p+2+i} &= [\by-\bx]_i, \qquad \forall i=1,\ldots,n.
	\end{align}
	\end{subequations}
	That is, $\t{\bv}$ is the same as $\bphi(\by)$ from \eqref{eq_min_frob_lagrange} except for the $t$-th entry.
	Specifically, we have
	\begin{align}
		\t{\bv} &= \bphi(\by) + \eta_t \bee_t,
	\end{align}
	where
	\begin{align}
		\eta_t &= \frac{1}{2}\left[(\by-\bx)^T (\by-\bx)\right]^2 - \frac{1}{2}\left[(\by-\bx)^T (\by_t-\bx)\right]^2 = \frac{1}{2} \|\by-\bx\|^4 - \bphi(\by)^T \bee_t.
	\end{align}
	Given this, \lemref{lem_row_col_update} gives us (denoting $\alpha_t := \bee_t^T F^{-1} \bee_t$ for notational convenience)
	\begin{align}
		\frac{\det(\t{F})}{\det(F)} &= \left(\bee_t^T F^{-1} \bphi(\by) + \eta_t \bee_t^T F^{-1} \bee_t\right)^2 + (\bee_t^T F^{-1} \bee_t) (\bphi(\by) + \eta_t \bee_t)^T (\bee_t - F^{-1} (\bphi(\by)+\eta_t\bee_t)), \\
		&= \left(\ell_t(\by) + \eta_t \alpha_t\right)^2 + \alpha_t\left(\bphi(\by)^T \bee_t - \bphi(\by)^T F^{-1} \bphi(\by) - \eta_t \ell_t(\by) + \eta_t - \eta_t \ell_t(\by) - \eta_t^2 \alpha_t\right), \\
		&= \ell_t(\by)^2 + \alpha_t\left(\bphi(\by)^T \bee_t - \bphi(\by)^T F^{-1} \bphi(\by) + \eta_t\right).
	\end{align}
	So from the definition of $\eta_t$, we get
	\begin{align}
		\det(\t{F}) &= \left(\ell_t(\by)^2 + \alpha_t\beta_t\right) \: \det(F),
	\end{align}
	where $\beta_t := \frac{1}{2}\|\by-\bx\|^4 - \bphi(\by)^T F^{-1} \bphi(\by)$.
	Finally, \cite[Lemma 2]{Powell2004} gives us $\alpha_t,\beta_t \geq 0$ provided $\ell_t(\by)\neq 0$, which gives us the result.\footnote{That $\ell_t(\by)\neq 0$ is not explicitly assumed in the statement of \cite[Lemma 2]{Powell2004} but is required as the proof relies on \cite[Theorem, p.~196]{Powell2004}.}
\end{proof}

\begin{remark}
	\thmref{thm_quad_update_det} does not require any of the points $\by_1,\ldots,\by_p$ or $\by$ to be in $C$.
	We will use this when constructing interpolation sets for which $F$ is invertible (\algref{alg_quad_init}).
\end{remark}

\begin{remark}
	A similar analysis to the above holds for linear interpolation ($p=n$), where \thmref{thm_quad_update_det} simplifies to $|\det(\t{F})| = |\ell_t(\by)| \cdot |\det(F)|$ as shown in \cite[Section 4.2]{LR_DFOGN2019}.
	This allows us to use an approach like \algref{alg_quad_poised} below to construct $\Lambda$-poised sets for linear interpolation, a simpler method than the one outlined in \cite[Section 4.3]{LR_ConvexDFO2021}.
\end{remark}

\subsection{Constructing $\bm{\Lambda}$-Poised Sets} \label{sec_constructing_poised_sets}
We begin by describing how to construct an initial interpolation set, contained in $B(\bx,\min(\Delta,1))\cap C$ for which $F$ is invertible.
The full process for this is given in \algref{alg_quad_init}.
We first generate an initial set $\{\by_1,\ldots,\by_p\}$ in $B(\bx,\min(\Delta,1))$, but not necessarily in $C$, for which $F$ is invertible.
We can do this using the approach outlined in \cite[Section 2]{Powell2009}, which chooses points of the form $\bx \pm \min(\Delta,1) \bee_t$ or $\bx \pm \min(\Delta,1) \bee_s \pm \min(\Delta,1) \bee_t$ for some $s,t\in\{1,\ldots,n\}$ in a specific way.
We then replace any of these points which are not in $C$ by new points in $B(\bx,\min(\Delta,1))\cap C$ while maintaining an invertibile $F$.

\begin{algorithm}[tb]
\begin{algorithmic}[1]
\Require Trust region center $\bx\in C$ and radius $\Delta>0$, number of interpolation points $p\in\{n+2, \ldots, (n+1)(n+2)/2\}$.
\vspace{0.2em}
\State Generate an initial collection of $p$ points $\{\by_1,\ldots,\by_p\} \subset B(\bx,\min(\Delta,1))$ using the method from \cite[Section 2]{Powell2009}. \label{ln_poised_init}
\While{there exists $t\in\{1,\ldots,p\}$ with $\by_t \notin C$}
	\State Calculate Lagrange polynomial $\ell_t$ and find $\by\in B(\bx,\min(\Delta,1))\cap C$ with $\ell_t(\by) \neq 0$. \label{ln_poised_replacement}
	\State Replace the interpolation point $\by_t$ by $\by$. 
\EndWhile
\State \Return interpolation set $\{\by_1,\ldots,\by_p\}$.
\end{algorithmic}
\caption{Construct an underdetermined quadratic interpolation set with invertible $F$ \eqref{eq_quad_interp_system}.}
\label{alg_quad_init}
\end{algorithm}

\begin{theorem} \label{thm_quad_init}
	For any $\bx\in C$, $\Delta>0$ and $p\in\{n+2, \ldots. (n+1)(n+2)/2\}$, the output of \algref{alg_quad_init} is an interpolation set $\{\by_1,\ldots,\by_p\} \subset B(\bx,\min(\Delta,1))\cap C$ for which $F$ \eqref{eq_quad_interp_system} is invertible.
\end{theorem}
\begin{proof}
	Firstly, \algref{alg_quad_init} terminates after at most $p$ iterations, since an index $t$ can only ever be selected once (after which the new $\by_t \in C$).
	
	Next, all points in the output set were either originally generated in line \ref{ln_poised_init}, in which case they must have originally been in both $B(\bx,\min(\Delta,1))$ (from the construction in line \ref{ln_poised_init}) and $C$ (in order to never be replaced in the main loop), or were a $\by\in B(\bx,\min(\Delta,1))\cap C$ chosen in line \ref{ln_poised_replacement}.
	Hence $\{\by_1,\ldots,\by_p\} \subset B(\bx,\min(\Delta,1))\cap C$.
	
	Finally, the initial collection $\{\by_1,\ldots,\by_p\}$ from \cite[Section 2]{Powell2009} generates an invertible $F$, with an explicit formula for $F^{-1}$ given in \cite[eqns.~2.11--2.15]{Powell2009}.
	Hence the initial $F$ satisfies $|\det(F)| > 0$.
	Then at each iteration replace $\by_t$ with a point $\by$ for which $\ell_t(\by) \neq 0$, and so the new $F$ also satisfies $|\det(F)| > 0$ by \thmref{thm_quad_update_det}.
	Thus the output interpolation set has an invertible $F$.
\end{proof}

\begin{algorithm}[tb]
\begin{algorithmic}[1]
\Require Trust region center $\bx\in C$ and radius $\Delta>0$, number of interpolation points $p\in\{n+2, \ldots, (n+1)(n+2)/2\}$, optional initial interpolation set $\{\by_1,\ldots,\by_p\}\subset C$, poisedness constant $\Lambda>1$.
\vspace{0.2em}
\If{initial interpolation set not provided or yields singular $F$ \eqref{eq_quad_interp_system} or $\max_{t=1,\ldots,p} \|\by_t-\bx\| > \min(\Delta,1)$} \label{ln_geom_check}
	\State Set $\{\by_1,\ldots,\by_p\}$ to be the output of \algref{alg_quad_init}.
\EndIf
\While{$\max_{\by\in B(\bx,\min(\Delta,1))\cap C} \max_{t=1,\ldots,p} |\ell_t(\by)| > \Lambda$} \label{ln_loop_start}
	\State Find $t\in\{1,\ldots,p\}$ and $\by\in B(\bx,\min(\Delta,1))\cap C$ with $|\ell_t(\by)| > \Lambda$.
	\State Replace the interpolation point $\by_t$ by $\by$ and recalculate Lagrange polynomials.
\EndWhile
\State \Return interpolation set $\{\by_1,\ldots,\by_p\}$.
\end{algorithmic}
\caption{Construct a $\Lambda$-poised underdetermined quadratic interpolation set.}
\label{alg_quad_poised}
\end{algorithm}

We can now present our algorithm to generate a $\Lambda$-poised set for minimum Frobenius norm interpolation, given in \algref{alg_quad_poised}.
This algorithm is very similar to existing methods for unconstrained linear/quadratic (e.g.~\cite[Algorithm 6.3]{Conn2009}) or constrained linear poisedness (\cite[Algorithm 4.2]{LR_ConvexDFO2021}).
Essentially we first ensure we have an interpolation set with invertible $F$ (so we can construct the associated Lagrange polynomials), and we then iteratively find indices $t$ and points $\by$ with $|\ell_t(\by)|>\Lambda$, replacing $\by_t$ with $\by$.

\begin{theorem} \label{thm_quad_poised}
	For any $\bx\in C$, $\Delta>0$, $p\in\{n+2, \ldots. (n+1)(n+2)/2\}$ and $\Lambda>1$, \algref{alg_quad_poised} produces an interpolation set $\{\by_1,\ldots,\by_p\}$ which is $\Lambda$-poised for minimum Frobenius norm interpolation in $B(\bx,\Delta)\cap C$ and for which $\|\by_t-\bx\| \leq \min(\Delta,1)$ for all $t=1,\ldots,p$.
\end{theorem}
\begin{proof}
	If \algref{alg_quad_poised} terminates, then by the definition of the loop check it must satisfy the requirements of the theorem, so it suffices to show that the `while' loop terminates.
	
	At the start of the `while' loop in line \ref{ln_loop_start}, we have that $\{\by_1,\ldots,\by_p\} \subset B(\bx,\min(\Delta,1))\cap C$ and $F$ is invertible (either because the `if' statement in line \ref{ln_geom_check} evaluates to false or by \thmref{thm_quad_init}).
	Hence at the start of the `while' loop we have $|\det(F)| > 0$.
	Let us call this value $d_0 > 0$.

	In each iteration of the `while' loop, we increase $|\det(F)|$ by a factor of at least $\Lambda^2$ by \thmref{thm_quad_update_det}.
	Hence after $k$ iterations of the loop, we have $|\det(F)| \geq \Lambda^{2k} d_0$.
	However, since by construction our interpolation set is contained in the compact set $B(\bx,\min(\Delta,1))\cap C$, there is a global upper bound $d_{\max}$ on $|\det(F)|$ (which is a continuous function of the interpolation points). 
	Hence the loop must terminate after at most $\log(d_{\max}/d_0) / (2\log\Lambda)$ iterations.
\end{proof}

\begin{remark}
    In the unconstrained case $C=\R^n$, \thmref{thm_quad_poised} and in particular its reliance on \thmref{thm_quad_update_det} clarifies the argument behind \cite[Theorem 6.6]{Conn2009}.
\end{remark}

Theorems~\ref{thm_fl_quad} (with $\beta=1$) and \ref{thm_quad_poised} then give us the following.

\begin{corollary}
	Suppose $f$ and $C$ satisfy Assumptions~\ref{ass_smoothness} and \ref{ass_feasible_set} respectively.
	Then for any $\bx\in C$, $\Delta>0$, $p\in\{n+2, \ldots. (n+1)(n+2)/2\}$ and $\Lambda>1$, \algref{alg_quad_poised} produces an interpolation set $\{\by_1,\ldots,\by_p\}$ for which the model $m$ generated by \eqref{eq_min_frob} is fully linear in $B(\bx,\Delta)$.
\end{corollary}

\section{Conclusions and Future Work}
In \cite{LR_ConvexDFO2021}, a weaker notion of fully linear models was used to construct convergent algorithms for \eqref{eq_main_problem} (namely \algref{alg_cdfotr}), and it was shown that linear interpolation models constructed only with feasible points can yield such models.
Here, we show that \algref{alg_cdfotr} can also be implemented using underdetermined quadratic interpolation models in the minimum Frobenius sense of \cite{Powell2004}, which is a much more useful construction in practice.
As an extra benefit, it also establishes the same results for linear regression models, as a strict generalization of the analysis in \cite{LR_ConvexDFO2021} that is independently useful in the case where only noisy objective evaluations are available.

The natural extension of this work is to the construction and analysis of second-order accurate models (``fully quadratric'' in the language of \cite{Conn2009}) built only using feasible points.
A concrete implementation of Algorithm~\ref{alg_cdfotr} with quadratic interpolation models would also be a useful way to extend the implementation from \cite{LR_ConvexDFO2021} (which only considers nonlinear least-squares objectives, for which linear interpolation models can be practically useful).

\addcontentsline{toc}{section}{References} 
\bibliographystyle{siam}
\bibliography{lr_pubs,refs} 


\end{document}